\documentclass{ksiam}
\usepackage{blindtext}
\usepackage{hyperref}
\usepackage{amsmath, amssymb, latexsym}
\usepackage{epsfig}
\DeclareMathOperator{\supp}{supp}

\newtheorem{axiom}{Axiom}[section]
\newtheorem{definition}{Definition}[section]
\newtheorem{assumption}{Assumption}[section]
\newtheorem{remark}{Remark}[section]

\newtheorem{theorem}{Theorem}[section]

\newtheorem{lemma}{Lemma}[section]

\AtBeginDocument{
\addtolength\abovedisplayskip{-0.2\baselineskip}
 \addtolength\belowdisplayskip{-0.2\baselineskip}
 \addtolength\abovedisplayshortskip{-0.2\baselineskip}
  \addtolength\belowdisplayshortskip{-0.2\baselineskip}
}

\newtheorem{corollary}{Corollary}[theorem]
\usepackage{setspace}
\usepackage[left=1 in,top=1 in,right=1 in,bottom=1 in]{geometry}
\begin{document}

\title[Finite speed of propagation in Einstein paradigm]
{Finite speed of propagation in Degenerate  Einstein Brownian Motion Model}

\author[I. G. Hevage]{Isanka Garli Hevage $^{\dag}$}
\address{Department of Mathematics and Statistics, Texas Tech University, TX 79409--1042, U. S. A.} 
\email{isankaupul.garlihevage@ttu.edu, akif.ibraguimov@ttu.edu}


\author[A. Ibraguimov]{ Akif Ibragimov}

\thanks{\authormark{\dag} Corresponding author.}


\begin{abstract}
We considered qualitative behaviour of  the generalization of Einstein's model of Brownian motion when the key parameter of the time interval of \textit{free jump} degenerates. Fluids will be characterised by number of particles per unit volume (density of fluid) at point of observation.  Degeneration of the  phenomenon manifests in two scenarios: a) flow of the fluid, which is highly dispersing like a non-dense gas and
b) flow of fluid far away from the source of flow, when the velocity of the flow is incomparably smaller than the gradient of the density.
First, we will show that both types of flows can be modeled  using the Einstein paradigm. 
We will investigate the question: What features will  particle flow exhibit if the time interval of the\textit{ free jump} is inverse proportional to the density and its gradient ?  We will show that in this scenario, the flow exhibits localization property, namely: if at some moment of time $t_0$ in the region, the gradient of the density or density itself is equal to zero, then for some $T$ during time interval $[ t_{0}, t_0 + T]$ there is no flow in the region. 
This directly links to Barenblatt's finite speed of propagation property for the degenerate equation. The method of the proof is very different from Barenblatt's method and based on the application of Ladyzhenskaya - De Giorgi iterative scheme and Vespri - Tedeev technique. From PDE point of view it assumed that solution exists in appropriate Sobolev type of space. 
\end{abstract}
\maketitle
\section{Introduction}\label{intro}
In his celebrated work from 1905,  Einstein   introduced frame work of the random motion of particles suspended in  media.
Einstein's paradigm took a pivotal role in this study  as we are  considering an extension of his thought experiment to explain localization properties which define many  physical problems. Einstein introduced three key parameters which characterize  random movement of the particles , i.e., non-colliding time interval ($\tau$), value of changes in the length  of particle displacement ($ \Delta$) which are ``free of collision" with other particles and media  during time interval $[t, t+\tau]$ and   frequency  ($\varphi(\Delta)$) of the  occurrence  of the ``free of collision" displacements $\Delta$ . In our work we interpret "non-collision displacements",  which we call hereafter as a \emph{free jump}. In order to be specific, we state all above in the form of  Einstein's Brownian motion axioms as a definition of \textit{free jump}.
\begin{definition} \normalfont
\begin{enumerate}\label{Einstein-assump}\
\item
    There exists a time interval $\tau$, which is very small compared to the time interval over which the system is observed, but large enough that the motions performed by a particle during two consecutive time intervals $\tau$ can be considered as mutually independent events. 
   \item
   Length of non-colliding  jumps  corresponding to time interval $\tau$, which we call ``free jumps" is  $\Delta$.
\end{enumerate}
\end{definition}
Traditionally physicists use free path to determine pass of non-collision of particles, as in \cite{free-pass}.
In this sense the term \textit{free jump} has the same meaning as free path. To make the definition of \textit{free jump} transparent,  we cite the Einstein's statement  as follows: ``\textit{Evidently it must be assumed that each single
particle executes a movement which is independent
of the movement of all other particles ; the
movements of one and the same particle after
different intervals of time must be considered as
mutually independent processes, so long as we
think of these intervals of time as being chosen
not too small.
We will introduce a time-interval $\tau$ in our discussion,
which is to be very small compared with
the observed interval of time, but, nevertheless,
of such a magnitude that the movements executed
by a particle in two consecutive intervals of time
r are to be considered as mutually independent
phenomena
}``. In order not to designate this two definition we use term \textit{free jump} instead of free path. Next, Einstein postulates that 
\begin{axiom}\label{conserv-law} \normalfont
At any point of observation $x$ at time $t+\tau$ total number of particles $u(x,t)$ in the unit volume $dv$ containing point $x$  is equal to accumulated  total number of particles which has free jumps $\Delta$ weighted  by frequency $\varphi(\Delta)$ from the same point $x$  at time $t$ 
 \begin{equation}\label{Einstein_conserv_eq}
 \begin{aligned}
     u(x, t+\tau) \cdot dv =  
\left[\int_{\mathbb{R}} u(x+ \Delta, t) \varphi(\Delta) d \Delta \right]\cdot dv .
 \end{aligned}
\end{equation} 
\end{axiom}
Regarding frequency Einstein also postulates the following  axioms.
\begin{axiom} {whole universes axiom:} \normalfont
\begin{equation}\label{univ ax}
\begin{aligned}
\int_{\mathbb{R}} \varphi(\Delta) d \Delta = 1. 
\end{aligned}
\end{equation}
\end{axiom}
\begin{axiom}\label{Even-Ein}{Evenness of the frequency:} \normalfont
\begin{align}
  \varphi(-\Delta)=\varphi(\Delta).
\end{align}
Therefore expected value of length in free jump $ \Delta_{e} = 0 $.
\end{axiom}
Assuming in the Axiom \ref{conserv-law}, evenness of the $\varphi(\Delta)$ and smoothness of  the  function $u(x,t)$ Einstein   approximates function $u(x,t)$ with the solution  of diffusivity  equation, which is often cited as  the Brownian motion equation 
\begin{equation}
    u_{t}  = D u_{xx}.\label{diff-eq}
\end{equation}
where $ \displaystyle D \triangleq \frac{1}{{\tau} }{\left[\int_{{\mathbb R}}  \frac{{\Delta}^{2}}{2} \varphi (\Delta) d\Delta\right]} $ 
is called Einstein's diffusion coefficient. If $D(x,t)$ is space and time dependent function and   ${1}/{c_{0}} > D > c_{0} > 0$ for some constant $c_0$ then, due to strong maximum principle, the solutions of equation \eqref{diff-eq} exhibit the so called important feature: " infinite speed of propagation". In many cases this can lead to  non-accurate  and non physical interpretations of the observations. In this article we will show that there exists an  extended Einstein paradigm of random motion such that solution of the corresponding equation will exhibit  finite speed of propagation. Namely, in our thought experiment (see Section \ref{eistein-paradigm}), we consider the case when duration of the time interval  $\tau$  of free jump  is inverse proportional to density $u$ (number of particles in unit volume) and its gradient $|\nabla u|$. We will show rigorously that this will lead to localisation property which in many sources is called  finite speed of propagation: \emph{If initial data has a compact support $K$, then for point $x_0\notin K$ there  exists $t_{x_0}$ such that solution $u(x_0,t)=0$ for $t\leq t_{x_0}$} (see \cite{EVAN}).\\
It is known that many of the physical phenomena in nature,  governed by diffusion and absorption, can be expressed by nonlinear parabolic partial differential inequalities, whose solution  exhibits finite speed of propagation properties. For example, in porous media if fluid is gas, then   perturbed on the well pressure will not immediately reach area away from source  of perturbation.  Corresponding diffusivity equation for pressure function has coefficients which degenerate as pressure vanishes (see \cite{Barenblatt-52},\cite{Barenblatt-96},\cite{Barenblatt-ss}). 
Another interesting example is the so called  pre-Darcy flow, which  occurs in far from well zone when the gradient of  pressure is so "small" that  particles of the liquid are not moving , i.e, has zero velocity (see \cite{pre-darcy-RRR},\cite{blosh},\cite{non-Newtonian}). In both   cases, from point of view of  Einstein paradigm the time interval of "free jump" of the particles composing fluid is infinitely big. 
\begin{remark}
It is worth  to mention that the number of particles per unit volume is a monotone function with respect to the density of the fluid, and correspondingly is monotone function with respect to the pressure for isothermic fluids.
\end{remark}
In the present work, we reconsider the above phenomena, and their governing nonlinear parabolic equations from the point of view of generalized Einstein's random walk model in a continuous medium with diffusion, drift and absorption or  reaction.\\
Starting from basic stochastic principles, we  derive a generic degenerating via solution and its gradient parabolic equation. This solution  models concentration of the particles in a unit volume.  We then  prove the strong localization property by hypothesising  the process  such that  in the Einstein framework  time interval of a "free" movement of particles depends on density particles  and its gradient. 
This approach can be interpreted as a complementary conceptual derivation of the governing equation for the flows, but without appealing to Fick's, Darcy's or Fourier's type laws, the continuity (or conservation of mass) equation, and thermodynamic closures that binds together the various variables of the system. 
We assume that the number of molecules in a liquid is proportional to the concentration of compound of interest,  characterized by a scalar function $u(x,t)$, which depends on the spatial and time coordinates $x$ and $t$, respectively.\\
This Article is organized as follows:
In the Section \ref{eistein-paradigm}  we  formally  introduce definition of random motion which includes: vector $\vec{\Delta} = (\Delta_1,\cdots, \Delta_N)$ of free jumps in $N$ dimensional space, its expected vector $\vec{\Delta_e}$ and some probability distribution $\phi(\vec{\Delta}) $. To derive our nonlinear model we use  the scalar $\tau$ as length of the time interval $[t,t+\tau]$ during which particles are not colliding. Generalizing the Einstein Axioms of mass conservation for the free jump process in  $N$ dimension , we introduce  "conservation of the mass" in \eqref{Einstein_conserv_eq-2} with  absorption or reaction and drift. Next, assuming that density function $u(x,t)$ is sufficiently smooth, we derive the PDE  inequality \eqref{M-1} for function $u(x,t)$ using Taylor's expansion and Caratheodory theorems.
Assuming the co-variances and the expected vector are $u$ independent, we model the process of "non-linear $\alpha-\beta$ jumps" by choosing $\tau$ to be inverse proportional  to density and its gradient in \eqref{tau-nonlin}. This axiom , alongside with assumption  on regularity of density function   allows reduction of conservation of  mass equation into a  nonlinear IBVP for differential inequalities. Then , by proposed method of Tedeev-Vespri  we use De - Giorgi-Ladyzhenskaya  construction  to map original problem for non-divergent differential inequality to  iterative integral inequalities  in Lemma \ref{lad-lemma-1} and Lemma \ref{lad-lemma-2}.  
Based on Lemmas \ref{lad-lemma-1} and Lemma \ref{lad-lemma-2}  in Theorem \ref{combined-ite} , we transform obtained inequalities  into one iterative inequality for specific functional . This Theorem enables us to implement Ladyzhenskaya Lemma to prove  one of the main result in Theorem \ref{I-0 bound} . Then , by Corollary \ref{col-FSP} we show that if initial data has compact support, the solution of the differential inequality \ref{model eq} will exhibit  finite speed of propagation property.
\section {Generalized Einstein paradigm} \label{eistein-paradigm}
In this section  we will extend model for Brownian Motion for dynamical process of  transport, diffusion and absorption with parameters, depending on number of particles and its gradient function. 
\subsection{Mass conservation law}\label{prelim}
Let  $x \in   \mathbb{R}^{N} $ and  $u(x, t)$ be the 
function which represent  the number of particles per unit volume at point $x$ and at time $t$. Consider a particle ($P$) of particular type suspended in the medium of interest. Denote $\mathbb{P}(\tau)$ to be the set of vectors with non-colliding jumps of $P$ corresponding to time interval $\tau$. We call  $\vec{\Delta} 
= \big(\Delta_1, \cdots ,\Delta_N\big)^{T}$ 
to be a '' vector of free jump of particles $P$'' if $\vec{\Delta} \in  \mathbb{P}(\tau)$.
We assume the following extension of the definition \ref{Einstein-assump}:
\begin{assumption} \label{ext-Eins}
\begin{enumerate}
\item All possible interactions between particles during time interval 
$\tau $ are via absorbing thorough surrounding media which may include media itself, other particles, and all possible boundaries. This key parameter $\tau$ in general can depend on the concentration of the particles ($P$) and its gradient and also space coordinate and time itself.  
\item\label{constitive} Time interval of free jumps $\tau$, expected vector $\vec{\Delta}_e$  of  a free jump $\vec{\Delta}$ and probability density function of free jump $\varphi(\vec{\Delta})$  are the only parameters  which characterise process of free jumps. Note that in a view of the definition of the set $\mathbb{P}(\tau)$, if $\vec{\Delta} \notin \mathbb{P}(\tau)$ then  $\varphi(\vec{\Delta})=0$. 
\item\label{absorbtion}
During time interval $[t,t+\tau]$ in the unit volume around the observation point $x$ ,  there are possible bonding and/or absorption with other particles or with the  media which  approximated by integral
\begin{align*}
    \int_t^{t+\tau}A(u(x, s)) \, ds  .
\end{align*}
It is important to state that if the non-linear function  $A(u) > 0, \ A(0)=0 $ then it has the growth constrain which will be introduced in Lemma \ref{lad-lemma-2}.
\end{enumerate}
\end{assumption}

\begin{axiom}  \normalfont
 Whole universe axiom:
 \begin{align}\label{uni-ax}
      \int_{\mathbb{P}(\tau)}\varphi(\vec{\Delta})d\vec{\Delta} = 1 .
 \end{align}
\end{axiom}
Let us define an expected vector of the "jumps" and corresponding co-variance matrix.
\begin{definition} \normalfont
\
\begin{enumerate}
\item 
Expected vector of free jumps 
\begin{equation}\label{exp-L}
\vec{\Delta}_e \triangleq (\Delta_{e}^{1},\Delta_{e}^{2},\dots,\Delta_{e}^{N})^{T} \quad   where \quad  \Delta_{e}^{i} \triangleq \int_{\mathbb{P}(\tau)}{\Delta_{i}} \varphi(\vec{\Delta})d\vec{\Delta} .
\end{equation}
\item Standard Co-variance matrix of a free jump
\begin{align}
\sigma_{ij}^2 \triangleq  \int_{\mathbb{P}(\tau)}\left({\Delta_{i}}-{\Delta}_{e}^{i}\right)\left({\Delta_{j}}-{\Delta}_{e}^{j}\right) \varphi(\vec{\Delta})d\vec{\Delta}\label{var-ij}.
\end{align}
\end{enumerate}
\end{definition}
Evidently $\vec{\Delta}_{e}(x,t)$ and $\sigma_{ij}(x,t)$ depend on space $x$ and time $t$. We postulate generalized Einsteins Axiom for the number of particles found at time $t+\tau$ in the control volume $dv$  contained point $x$ by 
\begin{axiom} \normalfont
\begin{equation}\label{Einstein_conserv_eq-2}
u(x, t+\tau) \cdot dv =  
\left[\int_{\mathbb{P}(\tau)} u(x+ \vec{\Delta}, t) \varphi(\vec{\Delta}) d \vec{\Delta} 
+ 
\int_t^{t+\tau}A(u(x, s),s)ds
\right]\cdot dv  .
\end{equation}
\end{axiom}
Mass conservation law \eqref{Einstein_conserv_eq-2} intuitively is easy to interpret, and it says that at any given point in space $x$ at time $t+\tau$ we will observe the number of particles per unit volume all particles with free jumps from the point $x$ at time $t$,  $+$ density of  particles which "produced" and $-$ density of  particles which "consumed" during time interval $[t,t+\tau].$  For comparison see \cite{Einstein56} (pages 14) first formula with integral.
\begin{remark}
Einstein definition of density of particles as a number of particles per unit volume differ from the fundamental definition of the density of fluid: denser the fluid, denser the number of particles in unit volume in the Einstein Paradigm.
\end{remark}

\subsection{Derivation of diffusion and absorption model with drift}\label{Derive}
Let $\zeta = (\zeta_1,\zeta_2,\cdots,\zeta_N)$ multi-index and  $ x^{\zeta} \triangleq x_{1}^{\zeta_1} \cdot x_{2}^{\zeta_2} \cdot \dots \cdot x_{N}^{\zeta_N}. $ Assuming that $u(x,t)\in C_{x,t}^{2,1} $ we apply Taylor's Expansion, and using \eqref{uni-ax} - \eqref{var-ij} we get  
\begin{align}
\hspace{-0.2 cm}\int_{\mathbb{P}(\tau) }
u(x + \vec{\Delta},t)\varphi(\vec{\Delta})d\vec{\Delta} = &  \nonumber \\
\text{ } u(x+\vec{\Delta}_{e},t) \ 
 + \ & 
\sum_{i\neq j} \sigma_{ij}^{2}  u_{x_{i}x_{j}}{(x+\vec{\Delta}_e,t)}
  + 
\frac{1}{2}\sum_{i =1} \sigma_{i}^{2}u_{x_{i}x_{i}}{(x+\vec{\Delta}_e,t)}  
  +  R_{\zeta}
\label{Taylor-eq}
\end{align} 
where \begin{align}
  R_{\zeta}  \triangleq 
\int_{\mathbb{P}(\tau) }
 \sum_{|\zeta| = 2}   
H_{\zeta}(x,\vec{\Delta},t)(\vec{\Delta}-\vec{\Delta_{e}})^{\zeta}
\varphi(\vec{\Delta})d\vec{\Delta}.\label{R-zeta}
\end{align}
Here
$\lim_{\vec{\Delta}\rightarrow\vec{\Delta_{e}} } H_{\zeta}(x ,\vec{\Delta},t) = 0$. Using \eqref{Taylor-eq} in \eqref{Einstein_conserv_eq-2} we get
\begin{equation}
u(x,t+\tau)- u(x+\vec{\Delta}_{e},t)  
\text{ = }  
 \sum_{i,j=1}^N  a_{ij}(x,t)  {u}_{x_{i} x_{j}} {(x+\vec{\Delta}_e,t)}\text{ + } R_{\zeta}
 \text{ + } 
\int_t^{t+\tau}A(u(x, s),t)ds.
\label{post-Taylor}
\end{equation}
Here
$ \displaystyle a_{ij}(x,t) = \frac{\sigma_{i}^{2}(x,t)}{2}  \text{ if }  i=j, \text{ and } a_{ij}(x,t) = \sigma_{ij}^{2}(x,t) \text{ if }  i\neq j $.
 Moreover, using Holder inequality for \eqref{R-zeta} with $0<l<1$ one can estimate
\begin{align}
R_{\zeta} \leq 
\sum_{|\zeta| = 2} \left[ \int_{\mathbb{P}(\tau)}|(\vec{\Delta}-\vec{\Delta_{e}})^{\zeta}
\varphi(\vec{\Delta})|^{\frac{1}{1-l}} d\vec{\Delta}
\right]^{1-l}  \cdot \left[\int_{\mathbb{P}(\tau) } 
|H_{\zeta}(x,\vec{\Delta},t)|^{\frac{1}{l}} d\vec{\Delta} \right]^{{l}}.
 \label{E-3}
\end{align}
Observe that LHS and RHS in the equation \eqref{post-Taylor} are defined in different points. In order to eliminate this ambiguity and derive the equation at the same point we will assume that $u(x,t) \in C^{3,2}_{x,t}.$ 
Then by Carathéodory's criterion    $\exists$ function $\vec{\psi}^{xxx}_{ij}: \mathbb{R}^{N} \times \mathbb{R}\rightarrow \mathbb{R}^N $  such that 
\begin{equation}
\sum_{i,j=1}^N a_{ij}(x,t) u_{x_{i}x_{j}} {(x+\vec{\Delta}_e,t)} =
\sum_{i,j=1}^N [ \vec{\psi}_{ij}^{xxx}(x,\vec{\Delta}_{e},t)\cdot\vec{\Delta}_{e}] a_{ij}(x,t) \text{ + }
\sum_{i,j=1}^N  a_{ij}(x,t)  u_{x_{i}x_{j}}{(x,t)}\label{E-2}.
\end{equation}
Similarly  $\exists $ functions $\psi^{t},\psi^{tt}, \psi^{x}_{i} \in \mathbb{R}$ and $ \vec{\psi}_{i}^{xx} \in \mathbb{R}^{N}$ such that 
\begin{multline}
u(x, t+ \tau)  - u (x+\vec{\Delta}_e, t) = \\
 [\tau^2 \psi^{tt} (x,t,\tau)]- \sum_{i=1}^{N}  [\vec{\psi}_{i}^{xx}(x,\vec{\Delta}_{e},t)\cdot \vec{\Delta}_{e}]\Delta_{e}^{i}  +
\tau \psi^{t} (x,t,0) 
 - \sum_{i=1}^{N}  \psi_{i}^{x}(x,0,t) \Delta_{e}^{i},
\label{E-1}
\end{multline}
where $\psi^{t} $ and $\psi_{i}^{x} $  are such that $  \lim_ {\nu \rightarrow 0} \psi^{t}(x,t,\nu) =u_{t}(x,t) $ and  $ \lim _{\Delta \rightarrow 0}\psi_{i}^{x}(x,\Delta,t) =  u_{x_{i}}(x,t) $. Moreover, in this article we will assume the term $ |\tau \psi_{tt} |$ negligible with respect to  $|\psi_t|$. Using  \eqref{E-1} in LHS of \eqref{post-Taylor}  and \eqref{E-3},\eqref{E-2}  in RHS of \eqref{post-Taylor} yields
\begin{align}
  u_{t}  - \frac{1}{\tau}\sum_{i=1}^{N}  \Delta_{e}^{i} u_{x_{i}} - \frac{1}{\tau} \sum_{i,j=1}^N   a_{ij}(x,t)  u_{x_{i}x_{j}} 
\leq 
\frac{B}{\tau}  + |A(u)| ,
\label{com-eq}
\end{align}
where
 \begin{equation}
B \triangleq \sum_{i=1}^{N}  [\vec{\psi}_{i}^{xx}(x,\vec{\Delta}_{e},t)\cdot \vec{\Delta}_{e}]\Delta_{e}^{i}
+
\sum_{i,j=1}^N [ \vec{\psi}_{ij}^{xxx}(x,\vec{\Delta}_{e},t)\cdot\vec{\Delta}_{e}] a_{ij}(x,t) 
     +  R_{\zeta} .
\end{equation}
Here $\vec{\Delta}_{e}$ and all functions are subject of growth condition with respect to   $|\nabla u|$. Further we will consider 
\begin{assumption} \label{B-bound}
 $\exists$  constant $C_1 \geq 0$ such that in R.H.S of inequality \eqref{com-eq} function 
 \begin{equation}
   B \leq   C_1|\nabla u(x,t)| , 
   \end{equation}
\end{assumption}
where $ \displaystyle \vert \nabla u \vert = \sqrt{\sum_{i=1}^{N} \bigg \vert \frac{\partial u}{\partial x_{i}} \bigg \vert^{2}}$.
This was introduced from general, mathematical point of view. For interpretation see \cite{lan-deadzone}. From above it follows that the function $u(x,t)$ can be approximated by the differential inequality
\begin{align}
   {u_t} - 
   \frac{1}{\tau}\sum_{i=1}^{N} b_{i}(x,t)  { u_ {x_{i}}} - 
   \frac{1}{\tau}  \sum_{i,j=1}^N &  a_{ij}(x,t) u_{x_{i}x_{j}}
    \leq 
\frac{1}{\tau} C_1|\nabla u| +|A(u)|, 
 \label{M-1} 
\end{align}
where $ b_{i}(x,t) = \Delta_{e}^{i}\label{b-i}$. 
 In forthcoming sections we will study a qualitative properties of the function $u(x,t)$ which solves inequality \eqref{M-1}, without refereeing to the origin of the this process which led to \eqref{M-1}. Obtained results then will provide interpretation for for physical processes which can lay under Einstein paradigm. 
 \begin{remark} 
 Partial differential inequality (PDI)  \eqref{M-1} has no unique solution for any initial data and boundary conditions but any solution which satisfies this inequality exhibits finite speed of propagation under the condition on $\tau$ in \eqref{tau-nonlin} with appropriately chosen $\alpha$ and $\beta$ in Lemma \ref{main-ineq-u} and Lemma \ref{lad-lemma-2}. In contrast to PDE s, PDI s has much wider application and corresponding dynamical systems will also exhibit localization property (see \cite{bro-tan},\cite{filip-diff}).\\
The PDI \eqref{M-1} has no unique solution but in-spite of that any solution of this inequality with  initial data having compact support will exhibit finite speed of propagation under our conditions on Non-linearity. RHS of \eqref{M-1} has two terms. The second term associated to the reaction term of Einstein type dynamic processes, whether first term is modeling non-linear drift, which bounded by $\nabla u$. Therefore inequality \eqref{M-1} has the generalization of the Einstein Brownian Motion model  not only due to  $\tau$, but also  because of the non-linear drift and reaction- absorption in the system. 
  \end{remark}
\subsection{Non-linear degenerate  partial differential inequality}\label{non-lin}
 In our model  $\tau$, $\vec{\Delta}_{e}$ and $\varphi$  are key characteristics of the process dynamics. Not only they can be functions of spatial and time variables  $x$ and $t$, properties of the medium and its boundary but also functions of dependent variables  and their derivatives. In this article we assume that  the process of the "free' jumps is  characterised only by the time interval of free jump ($\tau$). In our sought experiment we claim the process is such the time interval of free jumps $(\tau)$ is inverse proportional to number of particles $u(x,t)$  , and its gradient. This constrain can be intuitively justified by following arguments:\\
 a) Time interval of the free jumps of non-colliding particles increases as number of particles decreases.
 
 \begin{remark}
 Assuming the number of particles per unit volume to be proportional to density, then time interval of free jump will be reciprocal to density. In this way Einstein's paradigm can be used for characterization of highly "disperse" gases. It is known that the density of the disperse gases in porous media is proportional to pressure (see \cite{bar-Ent-Ryz}). This lead to so called porous media equation for pressure function, with degeneration in diffusion coefficient (see \cite{zelbar58}). Equation obtained from Einstein's thought experiment can be  mapped to a Barenblatt type equation for density function. This provides one more justification for assumption a) above. It is worth to mentioned that we also used Einstein's model with "free jumps" to interpret  the multi-component flow in the 1-dimensional tube, and provided method of the evaluation of the parameter $\tau$ based on spectrometer data (see \cite{Akif-Padgett}).   
 \end{remark}
 b) The Smaller the  gradient  of the density of particle lesser the  "mobility" of the fluid particles and, consequently the bigger the time interval of free jumps of non-colliding particles. 
 \begin{remark}
 To justify this assumption we  correlate the number of particles per unit volume through density to fluid pressure in porous media which will  relate Einstein's non-linear equation to so called pre-Darcy flow (see \cite{blosh},\cite{EIR}).
  \end{remark}
 The major question which we address in this article is following:
 \textit{Can the degeneration of the $\tau,$  with respect to solution or/and its gradient lead to localisation property of the particle-transport due to diffusion ?} To prove this localisation property we postulate the following dependence of the $\tau$ which stated in form of definition. 
\begin{definition}\label{tau} \normalfont
 Let $ \alpha > 0 $ and $ \beta \geq 0 $. The dynamical process, governs by  partial differential inequality \eqref{M-1} is called  the process of  $\alpha  - \beta$ jumps,  if  $\tau$ in \eqref{M-1} has the property
 \begin{equation}\label{tau-nonlin}
 \tau \approx \frac{1}{u^{\alpha}\cdot|\nabla u|^{\beta}}.
 \end{equation}
\end{definition}
In the recent paper \cite{CII20} in 1-D case, it was shown that if $\tau$ does not degenerate then qualitative property of Brownian motion is similar to linear case. 
In this article we consider the process of the jumps of particles when $\tau$ degenerates and  subject to constrain \eqref{tau-nonlin}. We assume that   variance matrix is homogeneous, positively defined and  isotropic: $  \sigma_{ij}(x,t) = 2 k_{2} \delta_{ij} $, where $k_2 > 0$, $\delta_{ij}$ is the  Kronecker symbol. Let
 \begin{align} 
{L}u \triangleq u_{t} -  u^{\alpha}|\nabla u|^{\beta} \sum_{i=1}^{N} b_i(x,t) {u_{ x_i}}-
{k_2}\cdot u^{\alpha}|\nabla u|^{\beta}\Delta u \label{l-u}
\end{align}
 Under the Definition \ref{tau} and  extended assumptions in assumption \ref{ext-Eins}, the differential inequality \eqref{M-1} with the corresponding initial and boundary condition in bounded domain $ \Omega \subset \mathbb{R}^N$, we defined  $u(x,t) $ as a non-negative classical solution of the following IBVP 
\begin{align} 
\begin{cases} 
 Lu \quad \leq \quad C_{1} u^{\alpha}|\nabla u|^{\beta+1} + |A(u)|  & \text{ in } \Omega\times (0,T]
\label{model eq} \\
u(x,0)    = u_0(x)    &\mbox{ in }  \Omega   \\
u(x,t)    = 0         & \mbox{ on }  {\partial \Omega \times (0,T]}.
  \end{cases}
 \end{align}
 \begin{remark}
 Note that the operator $L$ degenerates when $u=0$ or $|\nabla u|=0$, therefore its solution $u(x,t)$ does not belong to  $\mathcal{C}^{2,1}_{x,t}(\Omega\times (0,T])$, strictly speaking. 
 Our result is qualitative and does not address the existence of the solutions, but the obtained property of the solution will be applicable for the weak viscous solution which one can define as follows.
 Let  $ \epsilon \ \text{to be positive}$, and
\begin{align}
{L_{\epsilon}}u_{\epsilon} \triangleq (u_{\epsilon})_{t} -  ({u_{\epsilon}}^{\alpha}|\nabla u_{_{\epsilon}}|^{\beta}+ \epsilon) \sum_{i=1}^{N} b_i(x,t) ({u_{\epsilon})_{ x_i}}-
{k_2}\cdot ({u_{\epsilon}}^{\alpha}|\nabla u_{\epsilon}|^{\beta}+ \epsilon)\Delta u_{\epsilon}
\end{align}
then we define $u_{\epsilon}$ as a solution of the following IBVP$_{\epsilon}$ 
\begin{align}  \label{ibvp-epsilon}
\begin{cases} 
 L_{\epsilon}u_{\epsilon} \quad \leq \quad C_{1} {u_{\epsilon}}^{\alpha}|\nabla u_{\epsilon}|^{\beta+1} + |A(u_{\epsilon})|  & \text{ in }  \Omega\times (0,T]
 \\
u_{\epsilon}(x,0)    = u_0(x)    &\mbox{ in }  \Omega   \\
u_{\epsilon}(x,t)    = 0         & \mbox{ on } {\partial \Omega \times (0,T]}.
  \end{cases}
 \end{align}

In this article we showed that   all  estimates for the solution of the IBVP$_{\epsilon}$  with $L_{\epsilon}u_{\epsilon}=0$,  do not depend on $\epsilon$. All dependence on the regularization parameter $\epsilon$ appears as separate terms in the respective estimates, and they do not depend on the solution $u_{\epsilon}$ or its derivatives. This observation allow us to pass to the limit in the final estimates, and conclude the localization property for the limiting function
\begin{equation}\label{eps_sol}
u(x,t)=\lim_{\epsilon\to 0}u_{\epsilon}(x,t),
\end{equation}   
which is considered as a weak passage to the limit (see \cite{CIL92}). The obtained function $u(x,t)$ is called a weak viscosity solution of the IBVP, which will exhibit localisation property.  
\end{remark}

 \section{ Localization property of the solution IBVP } \label{localization}
The main goal of this article is to prove the Localization Property of the solution of IBVP \eqref{model eq} with $ \displaystyle \supp u_0 \subset B_{R_{0}}(0) = \{ \vert x \mid < R_{0}\}$. In order to prove in this section let us assume that the expected value of free jumps in each direction is uniformly bounded:
\begin{align}
|\Delta_{e}^{i}| \leq k_1 < \infty.      \label{bi-bound}
\end{align}
Then we will proceed with De Giorgi-Ladyzhenskaya iteration procedure as in \cite{ves-ted}, \cite{DGV}. 
Consider the sequence of $ \displaystyle  r_n = 
2r\left(1-\dfrac{1}{2^{n+1}}\right)$ for $ n=0,1,2,\dots$ with $r > 2R_0$. Let $ \bar{r}_n  = \left(\dfrac{r_n +r_{n+1}}{2}\right)$,  $\Omega_n = \Omega \setminus B_{r_n} (0)$ and $\Bar \Omega_n = \Omega \setminus B_{\Bar r_n} (0)$. 
Note that $ \Omega_{n+1} \subset \bar{\Omega}_{n} \subset \Omega_{n} \subset \Omega  $. Let $ 0 \leq \eta_{n} \leq 1$ be a sequence of cut off functions satisfying \begin{equation}\eta_n(x) = 0 \text{ for } x\in  B_{r_n}(0) ,\eta_n(x) = 1 \text{ for }  x\in \Bar \Omega_n  \text{ and } \mid \nabla \eta_n\mid \leq \frac{c2^n}{r} \text{ otherwise.} \label{cut-off}\end{equation}
\subsection{Preliminary Lemmas}\label{pre-lemmas}
\begin{lemma}\label{main-ineq-u}
Let $u(x,t)\geq 0$  be a classical solution of  IBVP \eqref{model eq}. Let $ \theta \geq 1 $ and  $p \geq 2$  be such that 
\begin{equation}
   \beta +2 \leq p <  \left(\frac{\theta + \alpha}{\beta+1}\right) - \frac{k_1}{k_2}-C_{1}\label{p-bounds}.
\end{equation}
Then
\begin{multline}
 \sup_{0 < \tau < t }\int_{\Omega_{n+1}} u ^{\theta +1} \, dx   \text{ + }    C \int_{0}^{t} {\int_{\Omega_{n+1}}} 
 u^{\theta +\alpha-1}  |\nabla u|^{\beta+2}  \, dx d\tau   
\\
\leq D_n \int_{0}^{t} \int_{\Omega_{n}} u^{\theta +\alpha+\beta+1} \, dxd\tau     
+
(\theta+1)\int_{0}^{t}\int_{\Omega_n}   |A(u)| u^{\theta} \, dxd\tau . 
\end{multline}
for $ 0 < t \leq T $. Here
\begin{align}
 C = (\theta +1)\left[k_2 \left( \frac{\theta + \alpha}{1+\beta}\right) -k_{1} -C_{1} -  k_{2} p\right]
, D_{n} =  (\theta +1)\left[ {k_2 p}\left(\frac{c 2^{n}}{r}\right)^{\beta +2} + {k_1} + C_{1} \right]\label{C-Dn} ,
\end{align}
\normalfont{\text{and} $C_{1}$ \text{is from Assumption} \ref{B-bound}.}
\end{lemma}
\begin{proof}
Let $\Omega_t \triangleq \Omega \times (0,t)$. Multiply both side of the inequality in \eqref{model eq} by $\eta_n^p u^{\theta}$, integrating by parts over $\Omega_t$ and using \eqref{bi-bound} we find
\begin{multline}
\frac{1}{\theta+1}\int_{\Omega} \eta_n^p u ^{\theta +1} \, dx - 
k_1\iint_{\Omega_t} \eta_n^p u^{\theta +\alpha} |\nabla u|^\beta 
|\nabla u |\, dx d\tau 
+\frac{k_2}{\beta +1}\iint_{\Omega_t}
\nabla\left(\eta_n^p u ^{\theta +\alpha}\right)|\nabla u|^\beta \nabla u  \ dx d\tau\\
\leq
C_{1}\iint_{\Omega_t} \eta_{n}^{p} u^{\alpha+\theta}|\nabla u|^{\beta+1} \,dxd\tau 
+
\iint_{\Omega_t} |A(u)|  \eta_{n}^{p}u^{\theta} \, dxd\tau.
\label{part-int}
\end{multline}
We compute $\nabla\left(\eta_n^p u ^{\theta +\alpha}\right) $ and \eqref{part-int} yields
\begin{multline}
\frac{1}{\theta+1}\int_{\Omega} \eta_n^p u ^{\theta +1} \, dx -k_1\iint_{\Omega_t} \eta_n^p u^{\theta +\alpha} |\nabla u|^{\beta+1}\, dx d\tau 
+\frac{k_2 (\theta +\alpha)}{\beta +1}\iint_{\Omega_t} \eta_n^{p} |\nabla u|^{\beta+2}  u ^{\theta +\alpha-1} \,dx d\tau \\
\leq
{k_2 p}\iint_{\Omega_t} \eta_n^{p-1}| \nabla \eta_n| |\nabla u|^{\beta+1} u ^{\theta + \alpha}  \, dx d\tau \\ 
+ C_{1}\iint_{\Omega_t} \eta_{n}^{p} u^{\theta+\alpha} |\nabla u|^{\beta+1} \, dxd\tau
+
\iint_{\Omega_t}  |A(u)|  \eta_{n}^{p}u^{\theta} \, dxd\tau   \label{E-6}.
\end{multline}
Apply Young's Inequity
    \begin{equation}
     \eta_n^{p-1}{}| \nabla \eta_n| |\nabla u|^{\beta+1} u ^{\theta +\alpha}
     \leq  
     \eta_n^{\frac{(p-1)(\beta+2)}{\beta+1}} |\nabla u|^{\beta+2}  u ^{\theta +\alpha-1} + | \nabla \eta_n|^{\beta+2} u^{\theta +\alpha+\beta+1}  ,
 \end{equation}
and
\begin{equation}
     |\nabla u|^{\beta+1} u^{\theta +\alpha} 
     \leq 
      |\nabla u|^{\beta+2}  u ^{\theta +\alpha-1} +  u^{\theta +\alpha+\beta+1}.
\end{equation}
Then estimate \eqref{E-6} becomes
\begin{multline}
\frac{1}{\theta+1} \int_{\Omega} \eta_n^p u ^{\theta +1} \, dx - k_1\iint_{\Omega_t} \eta_n^p u^{\theta +\alpha-1} |\nabla u|^{\beta+2}  +  \eta_n^p u^{\theta +\alpha + \beta+1}  \, dx d\tau \\
+ \frac{k_2 (\theta +\alpha)}{\beta+1}\iint_{\Omega_t} \eta_n^{p} |\nabla u|^{\beta+2}  u ^{\theta +\alpha-1}  \, dx d\tau \\
\leq
{k_2 p}
\iint_{\Omega_t} \eta_n^{{(p-1)}({\frac{\beta+2}{\beta + 1}})} u^{\theta +\alpha-1}  |\nabla u|^{\beta+2}   +   u^{\theta +\alpha+\beta+1} |\nabla \eta_n|^{\beta +2} \, dx d\tau\\
 + C_1\int_{0}^{t}\int_{\Omega_t}\eta_{n}^{p} u^{\theta +\alpha+\beta+1} +  \eta_{n}^{p} |\nabla u|^{\beta+2}  u ^{\theta +\alpha-1}  \, dxd\tau 
+
\int_{0}^{t}\int_{\Omega_t}   |A(u)|  \eta_{n}^{p}u^{\theta} \, dxd\tau. \label{post-young}
\end{multline}
Note that $ \left(  \frac{k_2(\theta + \alpha)}{1+\beta} -k_1 -C_{1}\right) \eta_n^{p} > {k_2 p} \eta_n^{{(p-1)}({\frac{\beta+2}{\beta + 1}})} $ by $ \eqref{p-bounds}$. Using \eqref{cut-off},  one can rearrange \eqref{post-young} to get 
\begin{multline}
\frac{1}{\theta+1} \int_{\Omega_{n+1}}   u ^{\theta +1} \, dx +  \int_{0}^{t}\int_{\Omega_{n+1}}
\left(  \frac{k_2(\theta + \alpha)}{1+\beta} -k_1  -C_{1}-  {k_2 p}  \right)   u^{\theta +\alpha-1}  |\nabla u|^{\beta+2} \, dx d\tau \\
\leq
{k_2}p \int_{0}^{t}\int_{\Omega_{n} \setminus {\bar{\Omega}_{n}}}  u^{\theta +\alpha+\beta+1}\left(\frac{c 2^{n}}{r}\right)^{\beta +2} \, dx d\tau
+ 
({k_1} + C_{1})\int_{0}^{t}\int_{{\Omega}_{n}}  u^{\theta +\alpha + \beta + 1} \, dx d\tau \\
+
\int_{0}^{t}\int_{\Omega_n}   |A(u)|  u^{\theta} \, dxd\tau .
\end{multline}
 So we will have the inequality
\begin{align} 
 L_{n}[u](t) \triangleq & \ \sup_{0 < \tau < t }\int_{\Omega_{n+1}}  u ^{\theta +1} \, dx    +    C \int_{0}^{t} {\int_{\Omega_{n+1}}}  u^{\theta +\alpha-1}  |\nabla u|^{\beta+2} \, dx d\tau \\
& \ \leq   D_{n} \int_{0}^{t} \int_{\Omega_{n}} u^{\theta +\alpha+\beta+1} \, dx d\tau
\ +
(\theta+1)\int_{0}^{t}\int_{\Omega_n}  |A(u)| u^{\theta} \, dxd\tau \\
\triangleq & \  \quad \quad \quad  \quad \quad  \quad  \ J_{n}[u](t)   \quad  \ \  \quad  \quad    +  \quad  \quad \quad  \quad \quad    K_{n}[u](t). \label{L-J-K} 
\end{align}
\end{proof} 
Next we introduce the following mapping
\begin{definition}  \label{L in I-n} \normalfont
Let
\begin{align}
 z \triangleq  u^{{(\theta +\alpha + \beta + 1)}/{(\beta +2)}}   \quad \text{ and } \quad  \lambda \triangleq{(\theta +1)(\beta +2)}/{(\theta +\alpha + \beta + 1)}.
 \label{z-u}
\end{align}
Define
\begin{equation}\label{I_n}
I_{n}[z](t) \triangleq \sup_{0 < \tau < t } \int_{\Omega_{n+1}} z^{\lambda} \, dx + \int_{0}^{t} \int_{\Omega_{n+1}} |\nabla z|^{\beta+2} \, dx d\tau, \quad  n = 0,1,2\dots \ .
\end{equation}
Using \eqref{z-u} one can get
\begin{align}
    {L_{n}}[u](t)=
    \sup_{0 < \tau < t } \int_{\Omega_{n+1}} z^{\lambda} \, dx \text{ + } C\int_{0}^{t} \int_{\Omega_{n+1}} \left(\frac{ \lambda}{\theta+1}\right)^{\beta+2} |\nabla z|^{\beta+2} \, dx d\tau  \triangleq \Tilde{L_{n}}[z](t).
\end{align}
Therefore
\begin{align}
  \Tilde{L_{n}}[z](t) 
    \geq
    G
\cdot I_{n}[z](t) ,
    \label{Lad-itn}
\end{align}
\end{definition}
where  $ \displaystyle G \triangleq \min 
\big(1,{ C{[\lambda/(\theta+1)]}^{\beta+2}
} \big)$.
Next in Lemma \ref{lad-lemma-1} and Lemma \ref{lad-lemma-2}, we will provide the relations for  $\Tilde{K_{n}}[z](t)$ and $\Tilde{J_{n}}[z](t)$ using $I_{n}[z](t)$. 
\begin{lemma}\label{lad-lemma-1}
Let $u(x,t)\geq 0$  be a classical solution of  IBVP \eqref{model eq} and $u$ and $z$ be as in \eqref{z-u}. Then $ \exists \text{ }  C_L,\epsilon_0 > 0$ and $  b_L > 1 $   such that 
\begin{equation}\label{Lad-itn-ineq-1}
\Tilde{J_{n}}[z](t) \leq t^{1-\Lambda} C_L b^{n-1}_L I_{n-1}^{1+\epsilon_0}[z](t) ,
\end{equation}
for any $n=1,2,\dots $ \ . Here  $ 0 < \Lambda < 1 $ and $\Tilde{J}_{n}[z](t)$  defined below in \eqref{bridge}.
\end{lemma}
\begin{proof}
By the substitution \eqref{z-u}
\begin{align}
J_{n}[u](t) \triangleq  D_{n} \int_{0}^{t} \int_{\Omega_{n}} u^{\theta +\alpha+\beta+1} \, dx d\tau
= 
{D_n}\int_{0}^{t} \int_{\Omega_{n}} z^{\beta +2} \, dx d\tau 
\triangleq
\Tilde{J_{n}}[z](t).
\label{bridge}
\end{align}
By Gagliardo-Nirenberg-Sobolev inequality (see \cite{LAU}) we obtain
\begin{align}
\int_{\Omega_{n}} z^{\beta+2}  \, dx 
& \leq c_G \left[ \int_{\Omega_{n}} |\nabla z|^{\beta+2} \, dx \right]^{\Lambda}\left[ \int_{\Omega_{n}} z^{\lambda} \, dx\right]^\frac{(1-\Lambda)(\beta+2)}{\lambda} ,
\label{N-G}
\end{align}
with \begin{equation}\label{zeta}
    \Lambda \triangleq [\alpha + \beta] \big / [\alpha + \beta +N(\beta +2)(\theta +1)] .   
\end{equation} 
Integrating above inequality over time and using Holder inequality, we get
\begin{align}
\int_{0}^{t}\int_{\Omega_{n}} z^{\beta+2}  \, dx d\tau 
& \leq
c_G t^{1-\Lambda} \left[ \int_{0}^{t} \int_{\Omega_{n}} |\nabla z|^{\beta+2} \, dxd\tau \right]^{\Lambda} d\tau
\left[\sup_{0 < \tau < t } \int_{\Omega_{n}} z^{\lambda} \, dx\right]^\frac{(1-\Lambda)(\beta+2)}{\lambda} .
\label{bridge2}
\end{align}
Using \eqref{bridge2} in \eqref{bridge} provides estimate
\begin{align}
\Tilde{J_{n}}[z](t)
\leq 
{D_n}
c_G t^{1-\Lambda} I_{n-1}^{\Lambda + \frac{(1-\Lambda)(\beta+2)}{\lambda}}[z](t)
\label{pre-ite} .
\end{align}
By \eqref{zeta}, let 
\begin{align}\label{epsilon-0}
    \epsilon_0 \triangleq N \Lambda (\beta+2) ,
\end{align}
 then $ \displaystyle {\Lambda + \frac{(1-\Lambda)(\beta+2)}{\lambda}}= 1 + \epsilon_0$.
Moreover by \eqref{p-bounds}  and \eqref{C-Dn} one can get
\begin{align}
D_n \leq (\theta +1) \left(k_1 + \left(\frac{c}{2R_{0}} \right) k_2 (\theta +\alpha)\right) 2^{{n\beta}+2n}\label{N-ineq} .
\end{align}
Therefore \eqref{pre-ite} will lead to
\begin{align}
    \Tilde{J_{n}}[z](t) \leq t^{1-\Lambda}C_Lb^{n-1}_L I_{n-1}^{1+\epsilon_0}[z](t)
\label{ite-sheme} .
\end{align}
Here $ 
C_L= (\theta +1) \left(k_1+ \left(\frac{c}{2R_{0}} \right) k_2 (\theta +\alpha)\right){c_G} 2^{{\beta}+2}
\quad \text{ and } \quad b_L =2^{\beta+2}$. 
\end{proof}
\begin{remark}\label{inf-speed-1}
Note that by \eqref{epsilon-0} 
\begin{align}
\epsilon_0 = 
[N(\alpha + \beta)(\beta+2)] \big / [\alpha + \beta +N(\beta +2)(\theta +1)] \label{epsilon - 1}    .
\end{align}
Let $ \alpha=0$ and $ \beta=0$ in $Lu=0$ in \eqref{l-u} which provides the non-Degenerate  parabolic equation  without absorption. Then  $\epsilon_0$ in \eqref{ite-sheme} will vanish. As we will see later, in this case Lemma \ref{Lady-lemma} will not provide Localization property as it intended. This reflects an important feature of the solution of parabolic equation which is called infinite speed of propagation. Indeed  due to strong maximum principle if $u(x,t)$ is non-negative solutions of second order-linear parabolic and  $ u(x_0,t_0) = 0 $ then it will vanish to zero on all subordinates to this point sub domain (see {\cite{lan-deadzone}}).
\end{remark}
\begin{lemma}\label{lad-lemma-2}
Let  $u$ and  $z$ be as in \eqref{z-u} .
\begin{enumerate}
\item [{\normalfont(P1)}]
Let $ \alpha > 0,\beta \geq 0$ and $N >1$  be such that $ \exists \text{ } \theta \geq 1 \text{ and } s \geq 1$, satisfy 
\begin{align}
    \max\left\{1+\epsilon_0, N\epsilon_0\right\}\leq s < \min\left\{\beta+2,N(1+\epsilon_0) \right\} ,  
    \quad \text{ where }\epsilon_{0} \text{ is from } \eqref{epsilon - 1} .
    \label{R-1} 
\end{align}
\vspace{-0.5 cm}
Define 
\begin{align}
 \gamma \triangleq \left(\frac{1+\epsilon_0}{s} - \frac{1}{\beta+2}\right)\lambda +1  ,\quad  \text{where } \lambda \text{ is from } \eqref{z-u} .	\label{gamma-def} \end{align}
\end{enumerate}
For given $A$, assume that  $\exists $ a transformation $F $ satisfying the following properties:
\begin{enumerate}
\setcounter{enumi}{2}
	\item [{\normalfont(P2)}]
	$ |A(u)| u^{\theta}  = F^{s}(z^{\gamma})$
	\item [{\normalfont(P3)}]
	$ F^{/} \leq M_{0} $ in $\Omega $. Here $^{./} \triangleq \frac{d}{dz}(\cdot)$ .
	\item [{\normalfont(P4)}]
	$F(z^{\gamma}(x)) \in  W_0^{1,m}(\Omega_n,\partial \Omega)$  where $ 1 \leq m < N $.
\end{enumerate} Then $ \exists
\text{ } M_L > 0 $  such that
\begin{align}
\Tilde{K_{n}}[z](t) \leq 
t^{\frac{\beta+2-s}{\beta+2}}M_Lb_L^{n-1} I_{n-1}^{1+ \epsilon_0}(t) .
\label{Lad-itn-ineq-2} 
\end{align}
 for any $n=1,2,\dots $ \ . Here $ \Tilde{K}_{n}[z](t)$ defined below in \eqref{K-t def} .
\end{lemma}
\begin{proof}
\text{Recall that } 
\begin{equation}\label{K-t def}
K_{n}[u](t) \triangleq  (\theta+1)\int_{0}^{t}\int_{\Omega_n}   |A(u)|  u^{\theta} \, dxd\tau  = (\theta+1)\int_{0}^{t}\int_{\Omega_n}  \mid F^{s}(z^{\gamma}) \mid \, dxd\tau \triangleq  \Tilde{K_{n}}[z](t)  .
\end{equation}
Let $ \displaystyle m \triangleq \frac{s}{1+\epsilon_0} $. By property \eqref{R-1} one has $ \displaystyle 1\leq s \leq  \frac{Nm}{N-m}$ . Consequently by Poincare-Sobolev inequality in (see \cite{EVAN}) 
\begin{align}
   \int_{\Omega_n} \mid F(z^\gamma)\mid^s \, dx
    \leq 
    \left(c_p {\gamma}\sup_{x \in \Omega_n} \mid F^{/} \mid\right)^{s}  
  \left [ \int_{\Omega_n}   z^{(\gamma-1)m}\mid \nabla  z\mid^{m} \, dx \right]^{1+\epsilon_{0}} .
  \label{poincare-ineq}
\end{align}
Using Holder Inequality we get
\begin{align}
     \left[ \int_{\Omega_n}   z^{(\gamma-1)m}\mid \nabla  z\mid^{m} \, dx \right]^{1+\epsilon_{0}}
    \leq
     \left[ \int_{\Omega_n} \mid \nabla z \mid ^{\beta+2} \, dx\right]^{(1+\epsilon_0)H}
     \cdot
     \left[ \int_{\Omega_n} z^{\lambda} \, dx \right]^{(1+\epsilon_0)(1-H)}
    \label{H-1} ,
\end{align}
with $ \displaystyle H \triangleq \frac{s}{[1+\epsilon_0][\beta+2]} $ . 
Note that $\displaystyle \frac{\gamma-1}{1-H} = \frac{\lambda}{m} $,   by \eqref{gamma-def} in Lemma \ref{lad-lemma-2}. Integrating  \eqref{poincare-ineq} over time and using \eqref{H-1} we obtain
 \begin{align}
 \int_{0}^{t}  \int_{\Omega_n} \mid F(z^\gamma)\mid^s \, dx d\tau
\leq 
   M 
 \int_{0}^{t} 
 \left( \int_{\Omega_n} \mid \nabla z \mid ^{\beta+2} \, dx\right)^{(1+\epsilon_0)H} d\tau
     \cdot
    \left[\sup_{0<\tau<t} \int_{\Omega_n}  z^{\lambda}\, dx \right]^{{(1+\epsilon_0)}{(1-H)}} ,
\end{align}
where $\displaystyle M^{1/s} \triangleq  c_p {\gamma}\sup_{(x,t) \in \Omega_n\times (0,t)} \mid F^{/} \mid$.
Applying Holder inequality in time 
 {\begin{align}
\int_{0}^{t}\int_{\Omega_n}  \mid F^{s}(z^{\gamma}) \mid \, dxd\tau
&\leq 
   Mt^{\frac{\beta+2-s}{\beta+2}}
      \left[\int_{0}^{t}   \int_{\Omega_n}\mid \nabla z\mid^{\beta+2} \, dx d\tau\right]^{\frac{s}{\beta+2}}\left[\sup_{0<\tau<t} \int_{\Omega_n}  z^{\lambda}\, dx \right]^{{(1+\epsilon_0)}{(1-H)}} .
      \label{H-2}
\end{align}}
Observe that  $ \displaystyle {{[1+\epsilon_0]}{[1-H]} + \frac{s}{\beta+2}} = 1 + \epsilon_0 $ . Then using \eqref{H-2} in \eqref{K-t def}  one has
  \begin{align}
  \Tilde{K_{n}}[z](t) \leq 
t^{\frac{\beta+2-s}{\beta+2}}M_Lb_L^{n-1} I_{n-1}^{1+ \epsilon_0}[z](t) ,
  \end{align}
where $ 
    M_L=(\theta+1)
 c_p^s{\gamma^s} \sup_{(x,t) \in \Omega_n\times (0,t)} \mid F^{/} \mid^s
$ .
\end{proof}
Combining the obtained inequalities  \eqref{Lad-itn}, \eqref{Lad-itn-ineq-1} and \eqref{Lad-itn-ineq-2} then \eqref{L-J-K} yields to generate the following iterative inequality.
\subsection{Localization property}\label{FSP}
\begin{theorem}\label{combined-ite}
Assume that all conditions in Lemmas \ref{main-ineq-u}, \ref{lad-lemma-1} and \ref{lad-lemma-2} are satisfied. Let $u(x,t)\geq 0$ be a classical solution of IBVP \eqref{model eq} and functions $u(x,t)$ and $z(x,t)$ be related as \eqref{z-u}. Then 
\begin{align}
    I_{n}[z](t) \leq  
     t^{q}
    \left(\frac{C_L +M_L}{G}\right)b^{n-1}_L  I_{n-1}^{1+ \epsilon_0}[z](t)  \label{I_n vs I_n-1}.
\end{align}
\end{theorem}
for $ 0 < t \leq T $ and  $n=1,2,\dots$ \ . Here  $t^{q} \triangleq \max \big\{  t^{1-\Lambda},t^{\frac{\beta+2-s}{\beta+2}} \big\} $,
$\Lambda $ is from \eqref{zeta} and $s$ is from \eqref{R-1}.
Next we prove localization property using the following generic Lemma by  Ladyzhenskaya in  \cite{LAU}.
\begin{lemma}\label{Lady-lemma}
Let sequence $y_n$ for  $n=0,1,2,...$ be non-negative, satisfying the recursion inequality
\[y_{n+1}\leq c\text{ }b^n \text{ }y_n^{1+\epsilon} ,\] for $n = 0,1 ,2,...$ with some positive constants $ c ,\epsilon > 0 \text{ and } b\geq 1$. Then \[    y_n \leq c^{\frac{(1+\epsilon)^n -1}{\epsilon}}\text{ } b^{\frac{(1+\epsilon)^n -1}{\epsilon^2} -\frac{n}{\epsilon}}\text{ }y_0^{(1+\epsilon)^n}.\]
In particular if $ y_0 \leq \theta_L = c^\frac{-1}{\epsilon} \text{ } b^\frac{-1}{\epsilon^2} \text{ and } b > 1  \text{ then } y_n \leq \theta_L\text{ } b^{\frac{-n}{\epsilon}} $. Consequently \[ y_n \rightarrow 0 \text{ when } n\rightarrow \infty . \]
\end{lemma} 
 From inequality \eqref{I_n vs I_n-1} and Lemma \ref{Lady-lemma}
it follows the main theorem on localization.
\begin{theorem}\label{I-0 bound}
Assume that all conditions in Theorem \ref{combined-ite} are satisfied. Let
\begin{equation}\label{I_oassump}
I_{0}[z](T) \leq {2^{-\left(\frac{\beta+2}{\epsilon_0^2}\right)}} {\left(\frac{G}{C_L+M_L}\right)^{\frac{1}{\epsilon_0}}}T^{-\frac{q}{\epsilon_0}}.
\end{equation}
\text{then }
\begin{equation}\label{In-to-0}
I_{n}[z](T)\rightarrow 0 \text{ as } n\rightarrow \infty .
\end{equation}
\end{theorem}
By the De- Giorgi construction we have  $  \  \Omega _{n+1}\subset \Omega_{n}$ . From above it follows that:
\begin{corollary}\label{col-FSP}
Let $u(x,t)\geq 0$ be a classical solution of IBVP \eqref{model eq} and $u$ and $z$ be related as in \eqref{z-u}. Assume as before   condition \eqref{I_oassump}. Then $ z(x,t) = 0 \text { a.e. in } \Omega_{\infty}=\bigcap_{i=n}^{\infty} \Omega_i \text{ for any } t \leq T$.
Consequently 
\begin{equation} 
u (x,t)= 0 \quad \text{ a.e. in } \quad  \Omega\setminus{B_{2r}(0)} \quad  \text{ for any } \quad   t\leq T.
\end{equation}
\end{corollary} 
Next, we can control value of the integral in \eqref{I_oassump} explicitly via initial data by assuming that in  IBVP \eqref{model eq}   reaction term to be equal zero. In this case  
\begin{equation}
\therefore \quad I_{0}[z](t) \leq 
\frac{D_{0}}{G} \int_{0}^{t} \int_{\Omega_{0}} z^{\beta +2} \, dx d\tau 
\leq \frac{D_{0}|\Omega|}{G} \Vert z\Vert_{L^{\infty}[(\Omega_{0})\times (0,t)]} ^{\beta+2} \cdot t .
\end{equation}
This estimate leads to the following 
\begin{theorem} \label{ max }
Let $ A (u) = 0 $ in IBVP \eqref{model eq} and assume that all conditions in Lemma \ref{main-ineq-u} are satisfied. Let $z(x,t)$ be such that in \eqref{z-u} and $ 
  \Vert z\Vert_{L^{\infty}[(\Omega_{0})\times (0,t)]}
 \leq 2^{-\frac{1}{\epsilon_{0}^2}} B_{0}^{\frac{1}{\beta+2}} t^{-{(\alpha + \beta + \theta + 1)}/{(\alpha + \beta)(\beta+2)}}
$ for $ 0 < t \leq T $.
Then
\begin{equation}\label{I0-estimate}
  I_{0}[z](t) \leq 2^{-\left(\frac{\beta+2}{\epsilon_{0}^2}\right)}\left(\frac{G}{C_L}\right)^{\frac{1}{\epsilon_{0}}} t^{-\frac{\theta+1}{\alpha+\beta}} \triangleq \theta_{L} .
\end{equation}
Here $\epsilon_{0}$ from \eqref{epsilon-0}, $ \displaystyle  B_{0} \triangleq \frac { C_{L}^{-\frac{1}{\epsilon_{0}}} G^{1+\frac{1}{\epsilon_{0}}}}{{D_{0}|\Omega_{0}|}} $ and $|  \cdot   |$ is the  size of domain.
\end{theorem}
\begin{corollary}
Let $u(x,t)\geq 0$ be a classical solution of IBVP \eqref{model eq}. Let the initial function $u_{0}(x)$ be such that  $\supp u_0(x)\in B_{R_0}(0)$ and 
\begin{equation}
    \Vert u_{0} \Vert_{L^{\infty}(\Omega)} \leq 2^{-\mu}\left[{\frac{ C_{L}^{-\frac{1}{\epsilon_{0}}} G^{1+\frac{1}{\epsilon_{0}}}}{D_{0}|\Omega|} }\right]^{\frac{1}{\alpha+\beta+\theta+1}} \cdot
    T^{-\frac{1}{\alpha+\beta}}
    \label{int-con} ,
\end{equation}
where $ \displaystyle \mu = \frac{[\alpha+\beta + N (\beta+2)(\theta+1) ]^2}{(\beta+2){N^2}{(\alpha+\beta)^2}(\alpha+\beta+\theta+1)}$. Then by the maximum principle function $z(x,t)$ which relate with $u(x,t)$ in \eqref{z-u}  satisfies $I_{0}[z](t) \leq \theta_{L}$.
Then due to Theorem \ref{I-0 bound}  if \eqref{int-con} holds then 
\begin{equation} 
 u (x,t)= 0 \quad \text{ a.e. in } \quad  \Omega\setminus{B_{2r}(0)} ,
\end{equation}
\end{corollary}
for $ r > 2R_{0}$  and  $ t\leq T $. Observe that the estimate \eqref{int-con} depends on volume  $|\Omega|$. Smaller the volume "bigger" initial data are allowed. Note that Vespri and Tedeeve (see \cite{ves-ted}) obtained boundedness of the solution of  degenerate parabolic equation without drift and right hand side (R.H.S.) for equation in divergence form without use of the maximum principle.
\begin{section}{Discussion}
Localization was first proposed by Zeldovich in 1950 (see \cite{zelbar58}) and then proved by Barenblatt by constructing  ``Barenblatt" type of barrier for corresponding degenerate non-linear PDE (see \cite{KV}).
Localisation property closely  relates to the property which in some sources is called a ``dead zone".  This property for solution of the steady state elliptic equation with strong absorption  was investigated in the work by Landis \cite{lan-deadzone} using his methods of lemma of growth for narrow domain. 
In the recent work by Vespri and Tedev method based on De- Giorgi (see \cite{degiorgi}) and Ladyzhenskaya iterative process (see \cite{LAU}) was employed to prove this essential feature for the class of degenerate parabolic equation in divergent form in \cite{ves-ted} . We use this as our groundwork and provide proof of the Localization property for degenerate parabolic equation in non-divergent form derived from generalized Einstein paradigm, which has a clear physical interpretation. It is appropriate to mention that Landis used his method to provide alternative prove of the De-Giorgi celebrated theorem  on Holder continuity of the solution of elliptic equation of second order (see \cite{Lan-el-par}). We believe that by employing Landis method we can significantly widen class of the equations for which localisation property holds by including absorption term.
\end{section}
\begin{section}{Conclusion}
Einstein paradigm was implemented for generalized Brownian motion process with drift and absorption or reaction.
We consider processes which allow implementation of Einstein paradigm for dynamic process  with   time interval of free jump to be inverse proportional to density and norm of its gradient.
The problem was reduced to a nonlinear partial differential inequality with drift and absorption or reaction.
We proved that this type of processes of $\alpha-\beta$ jumps , exhibits Localization  property.
To prove this Ladyzhenskaya- De Giorgi iterative procedure  was used.
Obtained result has very clear physical interpretation.
\end{section}



\begin{thebibliography}{99}
\bibitem{blosh} L. Bloshanskaya, A. Ibragimov, F. Siddiqui and M. Y. Soliman, {\it Productivity Index for Darcy and pre-/post-Darcy Flow (Analytical Approach)}, Journal of Porous Media, {\bf 20} (2017), 769--786.

\bibitem{Barenblatt-52} G. I. Barenblatt, {\it On some unsteady fluid and gas motions in a porous medium}, PMM journal of Applied Mathematics and Mechanics, {\bf16} (1952), 67--78.

\bibitem{Barenblatt-ss} G. I. Barenblatt, \textit{On self-similar motions of compressible fluid in a porous medium}, PMM journal of Applied Mathematics and Mechanics, {\bf16} (1952),679--698.

\bibitem{zelbar58} Y. B. Zeldovich and G. I. Barenblatt, \textit{On the asymptotic properties of self-similar solutions of the equations of unsteady gas filtration}, Doklady, USSR Academy of Sciences, {\bf118} (1958), 671--674.

\bibitem{CII20} I. C. Christov,  A. Ibraguimov, R. Islam, \textit{Long-time asymptotics of non-degenerate non-linear diffusion equations}, Journal of Mathematical Physics, {\bf61} (2020), 081505.

\bibitem{ves-ted} A. Tedeev and V. Vespri, \textit{Optimal behavior of the support of the solutions to a class of degenerate parabolic systems}, Interfaces and Free Boundaries, {\bf 17} (2015), 143--156.

\bibitem{lan-deadzone} E. M. Landis, \textit{On the dead zone for semilinear degenerate elliptic inequalities}, Differ Uravn, {\bf 29} (1993), 414--423.

\bibitem{KV} S. Kamin and J. L. Va'squez, \textit{Fundamental Solutions and Asymptotic Behaviour for the p-Laplacian Equation}, Revista Matematica Iberoamericana, {\bf 4} (1988), 339--354.

\bibitem{EIR} R. Ewing, A. Ibragimov and R. Lazarov, \textit{Domain decomposition algorithm and analytical simulation of coupled flow in reservoir/well system}, KSIAM, {\bf 5(2)} (2001), 71--99.

\bibitem{pre-darcy-RRR} R. Farmani, R. Azina and  R. Fateh, \textit{Analysis of Pre-Darcy flow for different liquids and gases}, Journal of Petroleum Science and Engineering, {\bf 168} (2018), 17--31.

\bibitem{degiorgi} E. de Giorgi, \textit{Sulla differenziabilità e l'analiticità delle estremali degli integrali multipli regolari}, Matematica, {\bf 4} (1960), 23--38.

\bibitem{NC} N. Kovalchuk  and C. Hadjistassou, \textit{Laws and principles governing fluid flow in porous media}, The European Physical Journal E, { \bf 42} (2019), 56.

\bibitem{bro-tan} B. Brogliato and A. Tanwani, \textit{Dynamical Systems Coupled with Monotone Set-Valued Operators: Formalisms, Applications, Well-Posedness, and Stability}, Society for Industrial and Applied Mathematics, {\bf 62(1)} (2020), 3--129.

\bibitem{Akif-Padgett} J. L. Padgett, Y. Geldiyev, S. Gautam, W. Peng, Y. Mechref and A. Ibraguimov, \textit{Object classification in analytical chemistry via data-driven discovery of partial differential equations}, Computational and Mathematical Methods, {\bf 3} (2021), e1164.

\bibitem{CIL92} M. G. Crandall, H. Ishii and P. L. Lions, \textit{User's guide to viscosity solutions of second order partial differential equations}, Bulletin of the American Mathematical Society,{ \bf 27} (1992), 1--67.


\bibitem{EVAN} L. C. Evans, \textit{Partial Differential Equations}, Graduate Studies in Mathematics {\bf 19}, American Mathematical Society, 2010.

\bibitem{Einstein56} A. Einstein, {\it Investigations on the Theory of the Brownian Movement}, edited by R. Fürth, Dover Publications, New York, translated by A.D. Cowper, 1956.


\bibitem{Barenblatt-96} G. I. Barenblatt, \textit{Scaling, Self-similarity, and Intermediate Asymptotics: Dimensional Analysis and Intermediate Asymptotics}, Cambridge Texts in Applied Mathematics {\bf 14}, Cambridge University Press, Cambridge, 1996.

\bibitem{LAU} O.A. Ladyženskaja,  V. A. Solonnikov and N. N. Ural'ceva, \textit{Linear and Quasi-linear Equations of Parabolic Type}, Translations of Mathematical Monographs {\bf 23}, American Mathematical Society, Providence,RI, 1968.

\bibitem{DGV} E. DiBenedetto, U. Gianazza and V. Vespri, \textit{Harnack's Inequality for Degenerate and Singular Parabolic Equations}, Springer Monographs in Mathematics {\bf 165} Springer, New York, 2012.

\bibitem{non-Newtonian} R.P. Chhabra and J.F. Richardson, \textit{Non‐newtonian flow in the process industries: Fundamentals and engineering applications}, Reed Educational and Professional Publishing Ltd, Oxford, UK, 1999.

\bibitem{Lan-el-par} E. M. Landis, \textit{Second Order Equations of Elliptic and Parabolic Type}, Translations of Mathematical monographs {\bf 171}, American Mathematical Society, Providence, RI, 1997.

\bibitem{filip-diff} A.F. Filippov, \textit{Differential Equations with Discontinuous Righthand Sides}, Mathematics and its Applications {\bf 18}, Springer, Netherlands, 1988.

\bibitem{bar-Ent-Ryz} G.I. Barenblatt, V.M. Entov and V.M. Ryzhik, \textit{Theories of Fluid Flows Through Natural Rocks}, Theory and Applications of Transport in Porous Media {\bf 3},	Dordrecht ; Boston : Kluwer Academic Publishers, 1990.

\bibitem{free-pass} Editors of Encyclopedia, \textit{Mean free path-Encyclopedia Britannica},2007. https://www.britannica.com/science/mean-free-path.
\end{thebibliography}
\end{document}